\documentclass[12pt]{amsart}

\usepackage{amsmath}
\usepackage{amsfonts}
\usepackage{amssymb}
\usepackage{amsthm}
\usepackage{newlfont}
\usepackage{url}
\usepackage{mathtools}
\usepackage{thmtools}
\usepackage{fullpage}
\usepackage[onehalfspacing]{setspace}
\usepackage[margin=20pt,textfont=footnotesize,labelfont=bf, justification=centerfirst, labelsep=endash]{caption}
\usepackage[colorlinks=true, pdfstartview=FitV, linkcolor=blue, citecolor=blue, urlcolor=blue]{hyperref}

\declaretheorem[style=definition,within=section,qed=$\blacktriangle$]{definition}
\declaretheorem[style=definition,qed=$\blacktriangle$,sibling=definition]{example}
\declaretheorem[style=plain,sibling=definition]{theorem}

\declaretheorem[style=plain,sibling=definition]{proposition}
\declaretheorem[style=plain,sibling=definition]{lemma}

\declaretheorem[style=plain,numbered=no,name=Maker's Ergodic Theorem]{makerstheorem}

\newcommand{\strc}{\text{s.c}}
\newcommand{\textG}{\text{G}}
\newcommand{\textB}{\text{B}}
\newcommand{\textW}{\text{W}}
\newcommand{\textd}{\text{d}}
\newcommand{\cale}{\mathcal{E}}
\renewcommand{\d}{~\text{d}}
\newcommand{\bbn}{\mathbb{N}}
\newcommand{\calf}{\mathcal{F}}
\newcommand{\calb}{\mathcal{B}}
\newcommand{\calr}{\mathcal{R}}
\newcommand{\bfr}{\mathbf{R}}
\newcommand{\bbz}{\mathbb{Z}}
\newcommand{\cald}{\mathcal{D}}
\newcommand{\cali}{\mathcal{I}}
\newcommand{\T}{\mathcal{T}}
\newcommand{\ep}{\epsilon}

\begin{document}

\title{Learning the ergodic decomposition}
\author{Nabil I. Al-Najjar}\address{Kellogg School of Management, Northwestern University}\email{al-najjar@kellogg.northwestern.edu}
\author{Eran Shmaya}\address{School of Mathematics, Tel Aviv University and Kellogg School of Management, Northwestern University}\email{erans@post.tau.ac.il}
\thanks{We thank Ehud Kalai, Ehud Lehrer and Rann Smorodinsky for helpful discussions.}
\date{First draft: February 2013;  This version:  \today
}

\maketitle

\begin{abstract}
A Bayesian agent learns about the structure of a stationary process from observing past outcomes. We prove that his predictions about the near future become approximately those he would have made if he knew the long run empirical frequencies of the process. 
\end{abstract}

\section{Introduction}
Consider a stationary, finite-valued stochastic process with probability law $\mu$. According to the ergodic theorem, an observer of this process can reconstruct the `true' ergodic component of the process from observing a single typical infinite realization. Decision problems, on the other hand, are often concerned with making predictions based on finite past observations.  In such problems, the primary object of interest is the predictive distribution about the outcome of the process at a given day given the finite history of outcomes from previous days. 

This paper relates these two perspectives on predictions and decisions. We consider the long-run properties of an observer's predictive distribution over next period's outcome as observations accumulate. We show that the predictive distribution  becomes arbitrarily close to the predictive distribution conditioned on knowledge of the true ergodic component, in most periods almost surely. Thus, as data accumulates, an observer's predictive distributions based on finite history become the `correct' predictions, in the sense of becoming nearly as good as what he would have predicted given knowledge of the objective empirical frequencies of the process. We demonstrate that the various qualifications we impose cannot be dropped.

Our results connect several literatures on learning and predictions in stochastic environments. First, there is the literature on the strong merging of opinions, pioneered by Blackwell and Dubins~\cite{BlackwellDubins:62}.%
\footnote{   Kalai and Lehrer~\cite{KalaiLehrer:93} apply this concept  to learning in games.}
 More directly relevant to our purpose are the weaker notions of merging introduced by Kalai and Lehrer~\cite{KalaiLehrer:94} and  Lehrer and Smorodinsky~\cite{LehrerSmorodinsky:96}, which focus on closeness of near-horizon predictive distributions. While strong merging obtains only under stringent assumptions, weak merging can be more easily satisfied. In our setting, for example, the posteriors may fail to strongly merge with the true parameter, no matter how much data accumulates. This strong notion of merging is unnecessary in contexts where decision makers discount the future or  care only about a fixed number of future periods. Weak merging, to which our results apply, is usually sufficient.

Another line of enquiry focuses on representations of the form $\mu=\int \mu_\theta\d\lambda(\theta)$, where a probability measure $\mu$ (the law of the stochastic process) is expressed as a convex combination of distributions $\{\mu_\theta\}_{\theta\in\Theta}$ that may be viewed as especially ``simple,'' or ``elementary.'' Such representations, also called 
\emph{decompositions}, are useful in models of learning where the set of parameters $\Theta$ may be viewed as the main object of learning. Two seminal theorems are de Finetti's representation of exchangeable distributions and the ergodic decomposition theorem for stationary processes.  Exchangeability rules out many interesting patterns of inter-temporal correlation, so it is natural to consider the larger class of stationary distributions. For this class, the canonical decomposition is in terms of the ergodic distributions. This is the finest decomposition possible using  parameters that are themselves stationary.  Our main theorem states that a Bayesian decision maker's predictions, based on finite histories, become arbitrarily close to those he would have made given knowledge of the true ergodic component.

Our result should also be contrasted with Doob's consistency theorem which states that Bayesian posteriors weakly converge to the true parameter. When the focus is the quality of decisions, what matters is not the agents' belief about the true parameter but the quality of his predictions.  Although the two concepts are related, they are not the same.  The difference is seen in the following example from Jackson, Kalai and Smorodinsky~\cite[Example 5]{JacksonKalaiSmorodinsky:99}: Assume that the outcomes Heads and Tails are generated by tossing a fair coin. If we take the set of all dirac measures on infinite sequences of  Heads-Tails outcomes as ``parameters'', then the posterior about the parameter converges weakly to a belief that is concentrated on the true realization. On the other hand the agent's predictions about next period's outcome is constant and never approach the predictions given the true ``parameter.''  This example highlights that convergence of  posterior beliefs to the true parameters may have little relevance to an agent's predictions and behavior.

Every process can be represented in an infinite number of ways, many of which, like the decomposition of the coin toss process above, are not very sensible.  Jackson, Kalai and Smorodinsky~\cite{JacksonKalaiSmorodinsky:99} study the question of what makes a particular decomposition of a stochastic process sensible. One requirement is for the process to be learnable, in the sense that an agent's predictions about near-horizon events become close to what he would have predicted had he known the true parameter. Given the close connection between ergodic distributions and long-run frequencies,  the most natural decomposition $\mu=\int \mu_\theta\d\lambda(\theta)$ of a stationary process is where the $\theta$'s index the  ergodic distributions. We show that their results do not apply to the class of stationary processes and their canonical ergodic decompositions. We show, however, that the ergodic decomposition is learnable in a weaker, yet meaningful sense as described below.

A third related literature, which traces to Cover~\cite{Cover:75}, is non-Bayesian estimation of stationary processes. See Morvai and Weiss~\cite{MorvaiWeiss:05} and the reference therein. This literature looks for an algorithm that make near-horizon predictions that are accurate for every stationary process. Our proofs of Theorem~\ref{th:thetheorem} and Example~\ref{ex:war} rely on techniques that were developed in this literature. There is however a major difference between that literature and our work: We are interested in a specific algorithm, namely Bayesian updating. 
Our agent's predictions and behavior are derived from this updating process.  We show how to apply the mathematical apparatus developed for the non-Bayesian estimation in our Bayesian setup.



\section{Formal model}

\subsection{Preliminaries}
An agent (a decision maker, a player,  or a statistician) observes a stochastic process $\left(\zeta_0,\zeta_1,\zeta_2,\dots\right)$ that takes values in a finite set of \emph{outcomes} $A$. Time is indexed by $n$ and the agent starts observing the process at $n=0$.  Let $\Omega=A^\bbn$ be the space of realizations of the process, with generic element denoted $\omega=(a_0,a_1, \ldots)$.  Endow $\Omega$  with the product topology and  the induced Borel structure $\calf$.  Let $\Delta(\Omega)$ be the set of probability distributions over $\Omega$. The law of the process is an element $\mu$ of $\Delta(\Omega)$.
A standard way to represent uncertainty about the process is in terms of  an index set of ``parameters:''
\begin{definition}
Let $\mu\in\Delta(\Omega)$. A \emph{decomposition}
  of $\mu$ is a quadruple 
  $\left(\Theta,\calb,\lambda,\left(\mu_{\theta}\right)\right)$ where:
  $\left(\Theta,\calb,\lambda\right)$ is a standard probability space of \emph{parameters} and $\mu_\theta\in\Delta(\Omega)$ for every $\theta\in\Theta$ such that the map $\theta\mapsto \mu_\theta(A)$ is $\calb$-measurable and 
\begin{equation}\label{eq:decomposition}
		\mu(A)=\int_\Theta \mu_\theta(A) ~\lambda(\textd\theta)
\end{equation}
for every $A\in\calf$. 
\end{definition}

A decomposition captures a certain way in which a Bayesian agent arranges his beliefs: The agent views the process as a two stages randomization. First a parameter $\theta$ is chosen according to $\lambda$ and then the outcomes are generated according to $\mu_\theta$. 
Beliefs can be represented in many ways. The two extreme decompositions are: (1) the {\em Trivial Decomposition.} with $\Theta=\{\bar\theta\}$, $\calb$ is trivial, and $\mu_{\bar\theta}=\mu$;  and (2) the {\em Dirac Decomposition.} with $\Theta=A^\bbn$, $\calb=\calf$, and $\lambda=\mu$. A ``parameter''  in this case is just a Dirac measure  $\delta_\omega$ that assigns probability 1 to the realization $\omega$. 

We are interested in decompositions that identify ``useful'' patterns shared by many realizations. These patterns capture our intuition of fundamental properties of a process. The two extreme cases are usually unsatisfactory. In the Dirac decomposition, there are as many parameters as there are realizations; parameters simply copy realizations. In the trivial decomposition, there is a single parameter and thus cannot discriminate between different interesting patterns.

Stationary beliefs admit a well-known decomposition with natural properties. Recall that the set of stationary measures over $\Omega$ is convex and compact in the weak$^\ast$-topology. Its extreme points are called \emph{ergodic beliefs}. We denote the set of ergodic beliefs by $\cale$. Every stationary belief $\mu$ admits a unique decomposition in which the parameter set is the set of ergodic beliefs: $\mu=\int \nu~\lambda(\textd \nu)$ for some belief $\lambda\in\Delta(\cale)$. We call this decomposition \emph{the ergodic decomposition}.

According to the ergodic theorem, for every stationary belief $\mu$ and every  block $(\bar a_0,\dots,\bar a_{k-1})\in A^k$, the limit frequency
\[\Pi(\omega;\bar a_0,\dots,\bar a_{k-1})=\lim_{n\rightarrow\infty}\frac{1}{n}\#\bigl\{0 \le t <n: a_{t}=\bar a_0,\dots,a_{t+k-1}=\bar a_{k-1}\bigr\}\]
exists for $\mu$-almost every realization $\omega=(a_0,a_1,\dots)$. When $\mu$ is ergodic this limit equals the probability $\mu([\bar a_0,\dots,\bar a_{k-1}])$. Thus, for ergodic processes, the probability of every block equals its (objective) empirical frequency.

The ergodic decomposition theorem states that for $\mu$-almost every $\omega$, The function $\Pi(\omega;\cdot)$ defined over blocks can be extended to a stationary measure over $\Delta(\Omega)$ which is also ergodic. Moreover, $\mu=\int \Pi(\omega;\cdot)~\mu(\textd \omega)$, so that the function $\omega\rightarrow\Pi(\omega;\cdot)$ recovers the ergodic parameter from the realization of the process. Thus, the parameters $\mu_\theta$ in the ergodic decomposition represent the empirical distribution of finite sequences of outcome along the realization of the stationary process. These parameters capture our intuition of fundamentals of the process.

A special case of the ergodic decomposition is the decomposition of an exchangeable distribution $\mu$ via i.i.d.\ distributions. For future reference, consider the following example:
\begin{example}\label{ex:dirichlet} 
The set of outcomes is $A=\{0,1\}$ and the agent's belief  is given by
	$$\mu\left(\zeta_n=a_0,\dots,\zeta_{n+k-1}
		=a_{k-1}\right)=\frac{1}{(k+1)\cdot {\binom{k}{d}}}$$
for every $n,k\in\bbn$ and $a_0,\dots,a_{k-1}\in A$ where $d=a_0+\dots+a_{k-1}$. Thus, the agent believes that if he observes the process $k$ consecutive periods then the number $d$ of good periods (periods with outcome $1$) is distributed uniformly in $[0,k]$ and all configuration with $d$ good outcomes are equally likely. 

De-Finetti's decomposition is given by $\left(\Theta,\calb,\lambda\right)$ where $\Theta=[0,1]$ equipped with the standard Borel structure $\calb$ and Lebesgue's measure $\lambda$, and, for $\theta\in\Theta$ $\mu_\theta\in\Delta(\Omega)$ is the distribution of i.i.d coin tosses with probability of success $\theta$:
	$$\mu_\theta\left(\zeta_n=a_0,\dots,\zeta_{n+k-1}
			=a_k\right)=\theta^d(1-\theta)^{k-d}$$
\end{example}

\subsection{Learning}
For every $\mu\in\Delta(\Omega)$ and sequence $\left(a_0,\dots,a_{n-1}\right)\in A^n$ with positive $\mu$-probability, the {\em $n$-period predictive distribution} is the element  $\mu(\cdot|a_0,\dots,a_{n-1}) \in \Delta(A)$ representing the agent's prediction about next period's outcomes given a prior $\mu$ and after observing the first $n$ outcomes of the process. Predictive distributions in this paper will always refer to one-step ahead predictions. This is for expository simplicity; our analysis covers any finite horizon.

Kalai and Lehrer \cite{KalaiLehrer:94}, and Kalai, Lehrer and Smorodinsky \cite{KalaiLehrerSmorodinsky:99} introduced the following notions merging. Note that in our setup, where the set of outcomes is the same in every period, this definition of merging is the same as `weak star merging' in~D'Aristotile, Diaconis and Freedman~\cite{DAristotileDiaconisFreedman:88}.
\begin{definition} Let
  $\mu,\tilde\mu\in\Delta(\Omega)$. Then 
the belief $\tilde\mu$ \emph{merges} to $\mu$ if 
	$$\left | \tilde\mu(\cdot|a_0,\dots,a_{n-1})-\mu(\cdot|a_0,\dots,a_{n-1})\right |
			\xrightarrow[n\rightarrow\infty]{\ }0$$
for  $\mu$-almost every realization $\omega=(a_0,a_1,\dots)\in A^\bbn$.

The belief $\tilde\mu$ \emph{weakly merges} to $\mu$ if
\begin{equation}\label{eq:awm}
	\left | \tilde\mu(\cdot|a_0,\dots,a_{n-1})-\mu(\cdot|a_0,\dots,a_{n-1})\right |\xrightarrow[n\rightarrow\infty]{\strc}0.%
	\footnote{ A bounded sequence of real numbers $a_0,a_1,\dots$ is said to \emph{strongly Cesaro converges} to a real number $a$, denoted $a_n\xrightarrow[n\rightarrow\infty]{\strc}a$, if $\lim_{n\rightarrow\infty}\frac{1}{n}\sum_{k=0}^{n-1}\left |a_k-a\right |=0$. Equivalently, $a_n\xrightarrow{\strc}a$ if there exists a full density set $T\subseteq\bbn$ such that $\lim_{n\rightarrow\infty, n\in T}a_n=a$.}
\end{equation}
for  $\mu$-almost every realization $\omega=(a_0,a_1,\dots)\in A^\bbn$.
\end{definition}
Here and later, for every pair $p,q\in\Delta(A)$ we let $\|p-q\|=\max_{a\in A}|p[a]-q[a]|$. These definitions were inspired by Blackwell and Dubins idea of strong merging, which requires that the prediction of $\tilde\mu$ will be similar to the prediction of $\mu$ not just for the next period but for the infinite horizon. 

\begin{definition}\label{de:learnable-rep}
A decomposition $\left(\Theta,\calb,\lambda,\left(\mu_{\theta}\right)\right)$ of $\mu\in\Delta(\Omega)$ is \emph{learnable} if $\mu$ merges with $\mu_\theta$ for $\lambda$-almost every $\theta$. The decomposition is {\em weakly learnable} if $\mu$ weakly merges with $\mu_\theta$ for $\lambda$-almost every $\theta$.
\end{definition}
As an example of a learnable decomposition, consider the Bayesian agent of Example~\ref{ex:dirichlet}. In this case 
	$$\mu(1|a_0,\dots,a_{n-1})=\frac{a_0+\dots+a_{n-1}+1}{n+1}.$$ 
For every $\theta\in [0,1]$ the strong law of large numbers implies that for every parameter $\theta\in[0,1]$ this expression converges  $\mu_\theta$-almost surely to $\theta$. Therefore $\mu$ merges with $\mu_\theta$ for every $\theta$, so De Finetti's decomposition is learnable (and, a fortiori, weakly learnable).  
This is a rare case in which the predictions $\mu(\zeta_n\in\cdot|a_0,\dots,a_{n-1})$ and $\mu_\theta(\zeta_n\in\cdot|a_0,\dots,a_{n-1})$ can be calculated explicitly. In general merging and weak merging are difficult to establish, because the Bayesian prediction about the next period is a complicated expression which is potentially depends on entire observed past.

\subsection{Motivation for Weak Merging}

In applications, $\mu$ represent the true process generating observations, and $\tilde\mu$ is a Bayesian agent's belief. To say that  $\tilde\mu$ weakly merges with $\mu$ means that  his next period predictions are accurate except for rare times. 

To connect this concept with statistical decision problems, suppose  that in every period, before the outcome is realized, the agent has to take some \emph{decision} from a finite set $D$. The agent's \emph{payoff} is represented by the  \emph{payoff function} $r:A\times D\rightarrow [0,1]$. 
A \emph{strategy} is given by $f:\bigcup_{n\ge 0}A^n\mapsto D$, with  $f(a_0,\dots,a_{n-1})$ denoting the action taken given the past realized outcomes. Let
	$$V_N(f)=\frac{1}{N}\int \sum_{n=0}^{N-1} r 
			\left 	(a_n,f(a_0,\dots,a_{n-1})\right)~\textd \mu$$
be the expected average expected payoff in the first $N$ periods. Fix $\ep>0$. A strategy $f^\ast$ is \emph{$\ep$-optimal for $N$ periods under $\mu$} if $V_N(f)\le V_N(f^\ast)+\ep$ for every strategy $f$. Of course the optimal strategy depends on the agent's belief $\mu$. The following proposition, which is immediate from the definition of weak learning, says that an agent who maximizes according to a belief that weakly merges with the truth will play $\ep$-optimal strategies against the truth if he is sufficiently patient. By `sufficiently patient' we mean that the horizon $N$ is large. Similar result applies if the agent aggregates periods' payoffs using some discount factor where by `sufficiently patient' is meant that the discount factor is close to $1$.
\begin{proposition}Let $\tilde\mu,\mu\in \Delta(\Omega)$ be such that $\tilde\mu$ weakly merges with $\mu$. For every $\ep>0$ there exists $N_0$ such that for every $N>N_0$, in every decision problem, every $0$-optimal strategy for $N$ periods under $\tilde\mu$ is $\ep$-optimal for $N$ periods under $\mu$.\end{proposition}

Kalai, Leher and Smorodinsky~\cite{KalaiLehrerSmorodinsky:99} provide a motivation for weak learning in terms of the properties of calibration tests. The idea of calibration originated with Dawid~\cite{Dawid:82}. A calibration test of a forecast compares the predicted frequency of events to their realized empirical frequencies. Kalai et al. showed that  $\tilde\mu$ weakly merges with $\mu$ if and only if forecasts made by $\tilde\mu$ pass all calibration tests of a certain type when the outcomes are generated according to $\mu$.

Finally, Lehrer and Smorodinsky~\cite{LehrerSmorodinsky:00} provide a characterization of weak merging in terms of the relative entropy between $\tilde\mu$ and $\mu$.\footnote{ However, we do not know whether their condition can be used to prove our theorem without repeating the whole argument.} No similar characterization is known for merging.

\subsection{Merging and the Consistency of Bayesian Estimators}

The idea of learning captured by Definition~\ref{de:learnable-rep} concerns the quality of  predictions made about near-horizon events. Another, perhaps more common, way to think about Bayesian inference is in terms of the consistency of Bayesian estimator. Consistency can be thought of as concerning learning the parameter itself. Recall that the Bayesian estimator of the parameter $\theta$ is the agent's conditional belief over $\theta$ after observing the outcomes of the process. It is well known that under any `reasonable' decomposition, the Bayesian estimator is consistent, i.e., the estimator weakly converges to the Dirac measure over the true parameter as data accumulates\footnote{ The argument traces back to Doob. See, for example, Weizsacker \cite{Weizsacker:96} and the references therein. It holds whenever the decomposition has the property that the realization of the process determines the parameter}. However, consistency of the estimator does not imply that the agent can use what he has learned to make predictions about future outcomes. For example, consider the Dirac decomposition of the process of fair coin tosses. Suppose the true parameter is $\omega^\ast$ for some $\omega^\ast=(\omega^\ast_0,\omega^\ast_1,\dots)$. After observing the first $n$ outcomes of the process the agent's belief about the parameter is uniform over all $\omega$ that agrees with $\omega^\ast$ on the first $n$ coordinates. While this belief indeed converges to $\delta_{\omega^\ast}$, the agent does not gain any new insight about the future of the process from learning the parameter. This decomposition is therefore not learnable.
\section{Main Theorem}
We are now in a position to state our main theorem.
\begin{theorem}\label{th:thetheorem}
The ergodic decomposition of every stationary stochastic process is weakly learnable.
\end{theorem}
To see the implications of our theorem, consider the following Hidden Markov process 
\begin{example}\label{ex:hidden}
An agent believes that the state of the economy every period is a noisy signal of an underlying ``hidden'' states that changes according to a Markov chain with memory 1. Formally, let $A=\{\textB,\textG\}$ be the set of outcomes, $H=\{\textB,\textG\}$ the set of hidden (unobserved) states, and $(\xi_n,\zeta_n)$ a $(H\times A)$-valued stationary Markov process with transition matrix $\rho:H\times A\rightarrow\Delta(H\times A)$ given by
	$$\rho(h,a)[h',a']=\bigl(p \delta_{h,h'}
			+(1-p)(1-\delta_{h,h'})\bigr)\cdot\bigl 
					(q \delta_{h',a'}+(1-q)(1-\delta_{h',a'})\bigr),$$
where $1/2<p,q<1$. Thus, if the hidden state in period $n$ was $h$ then at period $n+1$ the hidden state $h'$ remains $h$ with probability $p$ and changes with probability $1-p$. The observed state $a'$ of period $n+1$ will then be $h'$ with probability $q$ and is different from $h$ with probability $1-q$. Let $\mu_{p,q}\in \Delta(A^\bbn)$ be the distribution of $\zeta_0,\zeta_1,\dots$. Then $\mu_{p,q}$ is a stationary process which is not markov of any order. If the agent is uncertain about $p,q$ then his belief $\mu$ about the outcome process is again stationary, and can be represented by some prior over the parameter set $\Theta=(1/2,1]\times (1/2,1]$. This decomposition of $\mu$ will be the ergodic decomposition.
\end{example}
The consistency of the Bayesian estimator for $(p,q)$ implies that the conditional belief over the parameter $(p,q)$ converges almost surely in the weak-topology over $\Delta(\Theta)$ to the belief concentrated on the true parameter. However, because next-period's predictions involve complicated expressions that depend on the entire history of the process, it is not  clear whether these predictions merge with the truth. It follows from our theorem that they weakly merge.

Consider now the general case. If the agent knew the fundamental $\theta$, then at period $n$, after observing the partial history $(a_0,\dots,a_{n-1})$, his predictive probability that the next period outcome is $a_n$ would have been \begin{equation}\label{eq:nextperiod}\frac{\mu_\theta(a_0,\dots,a_{n-1},a_n)}{\mu_\theta(a_0,\dots,a_{n-1})}.\end{equation}Again consistency of the Bayesian estimator implies that, given uncertainty about the fundamental, the agent's assessment of $\mu_\theta(b)$ becomes asymptotically accurate for every block $b$. However, when the agent has to compute the next-period posterior probability~\eqref{eq:nextperiod}, he only had one observation of a block of size $n$ and no observation of the block of size $n+1$ so at that stage his assessment of the probabilities that appear in~\eqref{eq:nextperiod} may be completely wrong. Our theorem says that the agent would still weakly learn to make these predictions correctly.

Theorem~\ref{th:thetheorem} states that the agent will make predictions about near-horizon events as if he knew the fundamental of the process. Note, however, that it is not possible to ensure that the agent will learn to predict long-run events correctly, no matter how much data accumulates. For example, consider an agent who faces a sequence of i.i.d.\ coin tosses with parameter  $\theta\in[0,1]$ representing the probability of Heads. Suppose this agent has a uniform prior over [0,1]. This agent will eventually learn to predict near horizon outcomes as if he knew the true parameter $\theta$, but if he will continue to assign probability 0 to the event that the long-run frequency is $\theta$. In economic models, discounting implies that only near-horizon events matter.

We end this section with an example that in Theorem~\ref{th:thetheorem} weak learnability cannot be replaced by learnability. The example is a modification of an example given by Ryabko for the forward prediction problem in a non-Bayesian setup~\cite{Ryabko:88}. 
\begin{example}\label{ex:war}
Every period there is a probability $1/2$ for eruption of war. If no war erupts then the outcome is either bad economy or good economy and is a function of the number of peaceful periods since the last war. The function from the number of peaceful periods to outcome is an unknown parameter of the process, and the agent has a uniform prior over this parameter.

Formally, let $A=\{\textW,\textB,\textG\}$ be the set of outcomes. We define $\mu\in\Delta(A^\bbn)$ through its ergodic decompositions. Let $\Theta=\{\textB,\textG\}^{\{1,2,\dots\}}$ be the set of parameters with the standard Borel structure $\calb$ and the uniform distribution $\lambda$. Thus, a parameter is a function $\theta:\{1,2,\dots\}\rightarrow\{\textB,\textG\}$. We can think about this belief as a hidden markov model where the unobservable process $\xi_0,\xi_1,\dots$ is the time that elapsed since last time a war occurred. Thus, 
$\xi_0,\xi_1,\dots$ is the $\bbn$-valued stationary Markov process with transition probability
	$$\rho(j|k)=\begin{cases}&1/2,\text{ if }j=k+1,\\
			&1/2,\text{ if }j=0,\\
			&0,\text{ otherwise}.\end{cases}$$
for every $j,k\in\bbn$, and $\mu_\theta$ is the distribution of a sequence $\zeta_0,\zeta_1,\dots$ of $A$-valued random variables such that
	$$\zeta_n=\begin{cases}\textW, 
			&\text{ if }\xi_n=0\\\theta(\xi_n),
			&\text{ otherwise}.\end{cases}$$
Consider a Bayesian agent who observes the process. After the first time a war erupts the agent keeps track of the state of the process $\xi_n$ at every period. If there is no uncertainty about the parameter, i.e., if the Bayesian agent knew $\theta$, his prediction about the next outcome when $\zeta_n=k$ gives probability $1/2$ to outcome $\textW$ and probability $1/2$ to outcome $\theta(k+1)$. On the other hand, if the agent does not know $\theta$ but believes that it is randomized according to $\lambda$, he can deduce the values $\theta(k)$ gradually while he observes the process. However for every $k\in\{1,2,3,\dots\}$ there will be a time when the agent will observe $k$ consecutive peaceful period for the first time and at this point the agent's prediction about the next outcome will be $(1/2,1/4,1/4)$. Thus there will always be infinitely many occasions in which an agent who predicts according to $\mu$ will differ than an agent who predicts according to $\mu_\theta$. Therefore the decomposition is not learnable. On the other hand, in agreement with our theorem, these occasions become more infrequent as time goes by so the decomposition is weakly learnable. \end{example}

\section{Proof of Theorem~\ref{th:thetheorem}}\label{se:proof-thm}
Up to now we assumed that the stochastic process starts at time $n=0$. When working stationary processes it is natural to extend the index set of the process from $\bbn$ to $\bbz$, i.e. to assume that the process has infinite past. This is without loss of generality: every stationary stochastic process $\zeta_0,\zeta_1,\dots$ admits an extension $\dots,\zeta_{-1},\zeta_0,\zeta_1,\dots$ to the index set $\bbz$ \cite[Lemma 10.2]{Kallenberg:02}. We therefore assume hereafter, with harmless contrast with our previous notation, that $\Omega=A^\bbz$.  
 
Let $\cald$ be a $\sigma$-algebra Borel subsets of $\Omega$. The \emph{quotient space} of  $(\Omega,\calf,\mu)$ with respect to $\cald$ is the unique (up to isomorphism of measure spaces) standard probability space $(\Theta,\calb,\lambda)$ and a measurable map $\alpha:\Omega\rightarrow\Theta$  such that $\cald$ is generated by $\alpha$, i.e., for every $\calf$-measurable function $f$ from $\Omega$ to some standard probability space there exists a (unique up to equality $\lambda$-almost surely) $\calb$-measurable \emph{lifting} $\tilde f$ defined over $\Theta$ such that $ f=\tilde f\circ\alpha\quad\mu-\text{a.s.}$. The \emph{conditional distributions} of $\mu$ over $\cald$ is the unique (up to equality $\lambda$-almost surely) family $\mu_\theta$ of probability measures over $(\Omega,\calf,\mu)$ such that:
\begin{enumerate}
\item For every $\theta\in\Theta$ it holds that \begin{equation}\label{eq:conditional}\mu_\theta\left(\{\omega|\alpha(\omega)=\theta\}\right)=1.\end{equation}
\item The map $\theta\mapsto \mu_\theta(A)$ is $\calb$-measurable and~\eqref{eq:decomposition} is satisfied for every $A\in\calf$.
\end{enumerate}
We call $\left(\Theta,\calb,\lambda,\mu_\theta\right)$ the \emph{decomposition of $\mu$ induced by $\cald$}.
For every belief $\mu\in\Delta(\Omega)$, the trivial decomposition of $\mu$ is generated by the trivial sigma-algebra $\{\emptyset,\Omega\}$, the Dirac decomposition is generated by the sigma-algebra of all Borel subsets of $\Omega$. The ergodic decomposition  is induced by the $\sigma$-algebra $\cali$ of all \emph{invariant} Borel sets of $\Omega$, i.e. all Borel sets $S\subseteq\Omega$ such that $S=T^{-1}(S)$ where $T:\Omega\rightarrow \Omega$ is the left shift.

We will prove a more general theorem, which may be interesting in its own right. Let $T:A^\bbz\rightarrow A^\bbz$ be the left shift so that $T(\omega)_n = \omega_{n+1}$ for every $n\in\bbz$. A sigma-algebra $\cald$ of Borel subsets of $\Omega$ is \emph{shift-invariant} if $S\in \cald\leftrightarrow T(S)\in\cald$
for every Borel subset $S$ of $A^\bbz$.
\begin{theorem}\label{le:moregeneral}
Let $\mu$ be a stationary distribution over $\Omega$ and let $\cald$ be a shift invariant $\sigma$-algebra of subsets of $\Omega$ such that $\cald\subseteq \calf_{-\infty}^0$. Then the decomposition of $\mu$ induced by $\cald$ is weakly learnable.
\end{theorem}
Theorem~\ref{th:thetheorem} follows immediately from Theorem~\ref{le:moregeneral} since the sigma-algebra of invariant sets $\cali$ which induces the ergodic decomposition satisfies the assumption of the Theorem~\ref{le:moregeneral}.
We will prove Theorem~\ref{le:moregeneral} using Lemma~\ref{le:thelemma}
\begin{lemma}\label{le:thelemma}
Let $\mu$ be a stationary distribution over $A^\bbz$ and let $\cald$ be a shift invariant $\sigma$-algebra of Borel subsets of $A^\bbz$.  Then
\begin{equation}\label{thelemma-equation}
	\|\mu(\zeta_n=\cdot |
	\calf_0^n\vee\cald)-\mu(\zeta_n=\cdot |\calf_{-\infty}^n\vee\cald)\|
		\xrightarrow[n\rightarrow\infty]{\strc}0~\mu\text{-a.s}.
\end{equation}
\end{lemma}
Consider the case in which $\cald=\{\emptyset, \Omega\}$ is trivial. Then Lemma~\ref{le:thelemma} says that a Bayesian agent who observes a stationary process from time $n=0$ onwards will  make predictions in the long run as if he knew the infinite history of the process.
\begin{proof}[Proof of Lemma~\ref{le:thelemma}]
For every $n\ge 0$ let $f_n:\Omega\rightarrow\Delta(A)$ be a version of the conditional distribution of $\zeta_0$ according to $\mu$ given the finite history $\zeta_{-1} , \dots,\zeta_{-n}$ and $\cald$:
	$$f_n=\mu(\zeta_0=\cdot |\calf_{-n}^0\vee\cald),$$
and let $f_{\infty}:\Omega\rightarrow\Delta(A)$ be a version of the conditional distribution of $\zeta_0$ according to $\mu$ given the infinite history $\zeta_{-1},\dots$ and $\cald$: 
	$$f_{\infty}=\mu(\zeta_0=\cdot | \calf_{-\infty}^0\vee\cald).$$
Let $g_n=\|f_n-f_{\infty}\|$. By the martingale convergence theorem $\lim_{n\rightarrow\infty}f_n=f_{\infty}~\mu\text{-a.s}$ and therefore
\begin{equation}\label{martingale}
\lim_{n\rightarrow\infty}g_n=0~\mu\text{-a.s}
\end{equation}
It follows from the stationarity of $\mu$ and the fact that $\cald$ is shift invariant that 
\begin{equation}\label{shift}
\|\mu\left(\zeta_n=\cdot|\calf_0^n\vee\cald\right)-\mu\left(\zeta_n=\cdot|\calf_{-\infty}^n\vee\cald\right)\|=\|f_n\circ T^n-f_{\infty}\circ T^n\|=g_n\circ T^n~\mu\text{-a.s}\end{equation}
Therefore
	$$\frac{1}{N}\sum_{n=0}^{N-1}\|\mu
		\left(\zeta_n=\cdot|\calf_0^n\vee\cald\right)
			-\mu\left(\zeta_n=\cdot|\calf_{-\infty}^n\vee\cald\right)\|
		=\frac{1}{N}\sum_{n=0}^{N-1}g_n\circ T^n
			\xrightarrow[N\rightarrow\infty]{} 0~\mu\text{-a.s}$$
where the equality follows from~(\ref{shift}) and the limit follows from~\eqref{martingale} and Maker's generalization of the ergodic theorem~\cite[Corollary 10.8]{Kallenberg:02} to cover multiple functions simultaneously:
\begin{makerstheorem} Let $\mu\in \Delta(\Omega)$ be such that $T\mu=\mu$ and let $g_0,g_1,\dots:\Omega\rightarrow \bfr$ be such that $\sup_n |g_n|\in L^1(\mu)$ and $g_n\rightarrow g_\infty\quad\mu-a.s$. Then \[\frac{1}{N}\sum_{n=0}^{N-1}g_n\cdot T^n\xrightarrow[N\rightarrow\infty]{} E(g_\infty|\cali)\quad\mu-a.s.\]\end{makerstheorem} 
\end{proof}

\begin{proof}[Proof of Theorem~\ref{le:moregeneral}]
From $\cald\subseteq \calf_{-\infty}^0$ it follows that $\calf_{-\infty}^n\vee\cald=\calf_{-\infty}^n$.
Therefore, from Lemma~\ref{le:thelemma} we get that
	$$\|\mu(\zeta_n=\cdot |\calf_0^n\vee\cald)-\mu(\zeta_n=\cdot
		|\calf_{-\infty}^n)\|\xrightarrow[n\rightarrow\infty]
			{\strc}0~\mu\text{-a.s}.$$
By the same lemma (with $\cald$ trivial)
	$$\|\mu(\zeta_n=\cdot |\calf_0^n)-\mu(\zeta_n=
	\cdot|\calf_{-\infty}^n)\|\xrightarrow[n\rightarrow\infty]{\strc}0~\mu\text{-a.s}.$$
By the last two limits and the triangular inequality
\begin{multline}\label{induced-learnable}
	\|\mu(\zeta_n=\cdot |
	\calf_0^n)-\mu(\zeta_n=\cdot |
	\calf_0^n\vee\cald)\|\leq\\ \|\mu(\zeta_n=\cdot |
	\calf_0^n)-\mu(\zeta_n=\cdot
	|\calf_{-\infty}^n)\|+ \|\mu(\zeta_n=\cdot |
	\calf_0^n\vee\cald)-\mu(\zeta_n=\cdot
	|\calf_{-\infty}^n)\|\xrightarrow[n\rightarrow\infty]{\strc}0~\mu\text{-a.s}\end{multline}
Let $(\Theta,\calb,\lambda)$ be the quotient of $(\Omega,\calf,\mu)$ over $\cald$ and let $\left(\mu_\theta\right)$ be the corresponding conditional distributions. Let $S$ be the set of all realizations $\omega =(\dots,a_{-1},a_0,a_1,\dots)$ such that 
	$$\|\mu(\zeta_{n}=\cdot|a_{n-1},\dots,a_{0})-\mu_\omega \left(\zeta_{n}
			=\cdot|a_{n-1},\dots,a_{0}\right)\|\\ 
				\xrightarrow[n\rightarrow\infty]{\strc}0.$$
Then
$\mu(S)=1$ by~(\ref{induced-learnable}).  But
$\mu(S)=\int \mu_\theta(S) \lambda(\textd \theta)$. It follows that
$\mu_\theta(S)=1$ for $\lambda$-almost every $\theta$, a desired.
\end{proof}
\section{Ergodicity and mixing}
Mixing conditions formalize the intuition that observing a sequence of outcomes of a process does not change one's belief about events in the far future. Standard examples of mixing processes are i.i.d.\ processes and non-periodic markov processes. In this section we recall a mixing condition that was called ``sufficiency for prediction'' in JKS, show that the ergodic decomposition is not necessarily sufficient for prediction and show that a finer decomposition than the ergodic decomposition is sufficient for prediction and also weakly learnable.

Let $\overrightarrow\T=\bigwedge_{m\geq 0}\calf_{m}^\infty$ be the \emph{future tail} sigma-algebra where $\calf_m^\infty$ the $\sigma$-algebra of $\Omega$ that is generated by $(\zeta_m,\zeta_{m+1},\dots)$. A probability distribution (not necessarily stationary) $\nu\in\Delta(\Omega)$ is \emph{mixing} if it is $\overrightarrow\T$-trivial, i.e., if $\nu(B)\in\{0,1\}$ for every $B\in\overrightarrow\T$.\footnote{ An equivalent way to write this condition is that 
 for every $n$ and $\epsilon$, there is $m$ such that 
 		$$\left|\nu(B|a_0,\dots,a_{n-1})-\nu(B)\right|<\epsilon$$
 for every $B\in\calf_m^\infty$ and partial history $(a_0,\dots,a_{n-1})\in A^n$. JKS call such belief 
  {\em sufficient for prediction}. They establish the equivalence with the mixing condition in their proof of their Theorem 1}
If we want the components of the decomposition to be mixing we need a finer decomposition than the ergodic decomposition. This decomposition is the decomposition that is induced by the tail $\overrightarrow\T$ as shown in the following proposition.
\begin{proposition}Let $\left(\Theta,\calb,\lambda,\left(\mu_{\theta}\right)\right)$ be the decomposition of a belief $\mu\in\Delta(\Omega)$ that is induced by the tail $\overrightarrow\T$. Then $\mu_\theta$ is mixing for $\lambda$-almost every $\theta$.\end{proposition}
\begin{proof}This proposition is JKS' Theorem 1. We repeat the argument here to clarify a gap in their proof. 

The proposition follows from the fact that the conditional distributions of every probability distribution $\mu\in\Delta(\Omega)$ over the tail are almost surely tail-trivial (i.e., mixing). This fact was recently proved by Berti and Rigo~\cite[Theorem 15]{BertiRigo:07}\footnote{ It is taken for granted in the first sentence of JKS's proof of Their Theorem 1}. We note that it is not  true for every sigma-algebra $\cald$ that the conditional distributions of $\mu$ over $\cald$ are almost surely $\cald$-trivial. This property is very intuitive (and indeed, easy to prove) when $\cald$ is generated by a finite partition, or more generally when $\cald$ is countably generated, but the tail is not countably generated, which is why Berti and Rigo's result is required.
\end{proof}

The next theorem uses Lemma~\ref{le:thelemma} to show that the tail decomposition is also weakly learnable. In particular, Theorem~\ref{th:alsojks} implies that the ergodic decomposition does not capture all the learnable properties of a stationary process.
\begin{theorem}\label{th:alsojks}The tail decomposition of a stationary stochastic process is weakly learnable.\end{theorem}
\begin{proof}From Lemma~\ref{le:thelemma} it follows that the decomposition induced by the past tail $\overleftarrow\T$ is learnable, since the past tail is shift invariant. 

The theorem now follows from the fact that for every stationary belief $\mu$ over a finite set of outcomes it holds that $\overleftarrow\T_\mu=\overrightarrow\T_\mu$ where $\overleftarrow\T_\mu$ and $\overrightarrow\T_\mu$ are the completions of the past and future tails under $\mu$. See Weiss \cite[Section 7]{Weiss:00}. Therefore, the decomposition of $\mu$ induced by $\overrightarrow\T$ equals the decomposition induced by $\overleftarrow\T$, which is learnable. We note that the equality of the past and future tails of a stationary process is not trivial, it relies on finiteness of the set of outcomes $A$, and the proof relies on the notion of entropy. 
\end{proof}


We conclude with further comments on the relationship with \cite{JacksonKalaiSmorodinsky:99}. Their  main result characterizes the class of distributions that admit a decomposition which is both learnable and sufficient for prediction. They dub these processes ``asymptotically reverse mixing.'' In particular, they prove that, for every such process $\mu$, the decomposition of $\mu$ induced by the future tail is learnable and sufficient to prediction. In our Example~\ref{ex:war}, the tail decomposition equals the ergodic decomposition, and, as we have shown, is not learnable. This shows that stationary processes needs not be asymptotic reverse mixing.
On the other hand, the class of asymptotically reverse mixing processes contains non-stationary processes. For example, the Dirac atomic measure $\delta_\omega$ is asymptotically reverse mixing for every realization $\omega\in\Delta(\Omega)$.

\section{Extensions}
In this section we discuss to what extent the theorems and tools of this paper extend to a larger class of process. In the process, this sheds further light on the assumptions made in our work.

\subsection{Infinite set of outcomes}
The definitions of merging and weak merging can be extended to the case in which the outcome set $A$ is a compact metric space\footnote{ Also for the case that $A$ is a separable metric space, but then there are several possible non-equivalent definitions~\cite{DAristotileDiaconisFreedman:88}}: Let $\phi$ be the Prohorov Metric over $\Delta(A)$. Say that  
the belief $\tilde\mu\in\Delta(A^\bbn)$ \emph{merges} to $\mu\in\Delta(A^\bbn)$ if 
	$$\phi\left(\mu(\cdot|a_0,\dots,a_{n-1}),\mu(\cdot|a_0,\dots,a_{n-1})\right)			\xrightarrow[n\rightarrow\infty]{\ }0$$
for  $\mu$-almost every realization $\omega=(a_0,a_1,\dots)\in A^\bbn$ and that $\tilde \mu$ \emph{weakly merges} to $\mu$ if the limit holds in strong Cesaro sense. Theorem~\ref{th:thetheorem} extends to the case of an infinite set $A$ of outcomes. However, Theorem~\ref{th:alsojks} does not hold in this case. We used the finiteness in the proof when we appealed to the equality of the past and future tails of the process. The following example shows the problem where $A$ is infinite:
\begin{example}\label{ex:infinite}Let $A=\{0,1\}^\bbn$ equipped with the standard Borel structure. Thus an element $a\in A$ is given by $a=\left(a[0],a[1],\dots\right)$ where $a[k]\in\{0,1\}$ for every $k\in\bbn$. Let $\mu$ be the belief over $A^\bbz$ such that $\{\zeta_n[0]\}_{n\in\bbz}$ are i.i.d.\ fair coin tosses and $\zeta_n[k]=\zeta_{n-k}[0]$ for every $k\geq 1$. Note that in this case $\overrightarrow\T=\calb$ (so the future tail contains the entire history of the process) while $\overleftarrow\T=\calr$ (the past tail is empty). The tail decomposition in this case will be the Dirac decomposition. However, this decomposition is not learnable: an agent who predict according to $\mu$ will at every period $n$ will be completely in the dark about $\zeta_{n+1}[0]$.\end{example}
\subsection{Relaxing stationarity}
As we have argued earlier, stationary beliefs are useful to model situations where  there is nothing remarkable about the point in time in which the agent started to keep track of the processes (so other agents who start observing the process at different times have the same beliefs) and that the agent is a passive observer who has no impact on the process itself. The first assumption is rather strong, and can be somewhat relaxed. In particular, consider a belief that is the posterior of some stationary prior conditioned on the occurrence of some event. 
(A similar situation is an agent who observes a finite state markov process that starts at a given state rather than the stationary distribution.)  Let us say that a belief $\nu\in A^\bbn$ is \emph{conditionally stationary} if there exists some stationary belief $\mu$ such that $\nu=\mu(\cdot |B)$ for some Borel subset $B$ of $A^\bbn$ with $\mu(B)>0$. While such processes are not stationary, they still admits an ergodic decomposition. they exhibit the same tail behavior of stationary processes. In particular, our theorems extend to such processes. The obvious details are omitted.

\bibliographystyle{plain}
\bibliography{references}

\end{document}